\documentclass[a4paper,draft]{amsproc}
\usepackage{amssymb}
\usepackage{amscd} 
\usepackage{graphicx} 

\usepackage[alphabetic,y2k,initials]{amsrefs}
\usepackage{pst-all}
\usepackage{pstricks-add}
\usepackage{pst-slpe}
\usepackage{pst-solides3d}

\theoremstyle{plain}
 \newtheorem{thm}{Theorem}[section]
 \newtheorem{prop}{Proposition}[section]
 
 \newtheorem{cor}{Corollary}[section]
\theoremstyle{definition}

\theoremstyle{remark}
 \newtheorem{rem}{Remark}[section]
 \numberwithin{equation}{section}

\renewcommand{\le}{\leqslant}
\renewcommand{\ge}{\geqslant}

\title[Topology of the elliptical billiard with the Hooke's potential]{Topology of the elliptical billiard with the Hooke's potential}

\subjclass[2010]{Primary 37J35; Secondary 70H06}

\keywords{elliptical billiards, confocal conics, integrable potential perturbations, Hooke's potential field, Fomenko graphs}

\author[Radnovi\'c]{\bfseries Milena Radnovi\'c}

\address{
Mathematical Institute SANU, Kneza Mihaila 36, Belgrade, Serbia
\newline\indent School of Mathematics and Statistics, University of Sydney, NSW, Australia}
\email{milena@mi.sanu.ac.rs, milena.radnovic@sydney.edu.au}

\begin{document}

 \setcounter{page}{1} \thispagestyle{empty}

\begin{abstract}
Using Fomenko graphs, we present a topological description of the elliptical billiard with Hooke's potential.
\end{abstract}


\maketitle
\tableofcontents

\section{Introduction}\label{sec:intro}
Mathematical billiard is a system where a material point moves with a constant velocity inside the billiard desk, while the impacts off the boundary are ideally elastic and the velocity is changed according to the billiard reflection law \cite{KozTrBIL}.
The boundary of the billiard desk can be any piecewise smooth curve.

A remarkable class of such systems are elliptical billiards, due to their nice geometrical properties.
Each trajectory of elliptical billiard has a \emph{caustic} -- a curve touching all segments, or their straight extensions, of the trajectory.
Moreover, the caustic is a conic confocal with the boundary and it represents a geometrical manifestation of the integral of motion \cites{ArnoldMMM,KozTrBIL,DragRadn2011book}.

One can consider billiard motion in a potential field as well.
Integrable classes of potential perturbations of elliptical billiards were studied in \cites{Kozlov1995,Drag1996,Drag2002}.
The Hooke's potential belongs to that class and the dynamics of such a billiard was studied in \cite{Fed2001}, where the Lax representation of that system was constructed and a general solution obtained in terms of theta-functions.
Some remarkable geometrical properties of the system are also proved in \cite{Fed2001}: each trajectory has two caustics which are confocal with the boundary.

The objective of this work is to give topological description of the elliptical billiard with the Hooke's potential.
To achieve that, we are using tools, widely known as the Fomenko graphs, which codify  the topological essence of the Liouville foliation of integrable systems with two degrees of freedom.
The detailed account on Fomenko graphs can be found in \cite{BolFomBOOK}, see also \cites{BMF1990,BO2006,BBM2010} and references therein.
For topological description of elliptical billiards without potential see \cites{DragRadn2009,DragRadn2013publ,DragRadn2015fom} and for billiards within confocal parabolas \cite{Fokicheva2014}.

Next, in Section \ref{sec:integration}, we review the main steps from \cite{Fed2001} of the integration procedure of the elliptical billiard within an ellipse in the presence of the Hooke's potential.
Section \ref{sec:topology} contains topological description of that system.
The main result is Theorem \ref{th:fomenko}, where the Fomenko ivariants for the isoenergy manifolds are calculated.
For the critical closed orbits, we discuss their stability in Corollary \ref{cor:stability}.

\section{Integrability of the system}\label{sec:integration}
In \cite{Fed2001}, billiard within an ellipsoid in the $n$-dimensional space with the Hooke's potential was considered.
In this section, we review the integration procedure from \cite{Fed2001} in the planar case: $n=2$.

Suppose that in the field with the \emph{Hooke's potential}:
\begin{equation}\label{eq:potential}
\frac{\sigma}2(x^2+y^2),\quad \sigma>0,
\end{equation}
a particle of unit mass moves within ellipse:
\begin{equation}\label{eq:E}
\mathcal E\ :\ \frac{x^2}{a}+\frac{y^2}{b}=1,
\quad
a>b>0,
\end{equation}
satisfying the billiard law when being reflected off the boundary.

Let $\mathcal B\ :\ (\xi,v)\mapsto(\tilde\xi,\tilde v)$ be the billiard
mapping, where $\xi, \tilde\xi\in\mathcal E$ are consecutive points of
reflection off the boundary, and $v,\tilde v$ the velocity vectors at
$\xi,\tilde\xi$ after the reflection.
Explicit formulae for $\mathcal B$ are:
\begin{gather*}
\tilde\xi =
 - \frac1{\nu} \big(\sigma\xi-(v,Av)\xi+2(\xi,Av)v\big),
  \\
\tilde v =
 - \frac1{\nu} \big(\sigma v-(v,Av)v-2\sigma(\xi,Av)\xi\big)
 + \mu A\tilde\xi,
  \\
\nu=\sqrt{4\sigma(\xi,Av)^2+(\sigma-(v,Av))^2},
\quad
\mu=\frac{2(\tilde v,A\tilde\xi)}{(\tilde\xi,A^{2}\tilde\xi)},
\quad
A=\left(\begin{array}{cc}
         1/a & 0\\
         0 & 1/b
        \end{array}
\right).
\end{gather*}

It turns out that the billiard mapping $\mathcal B$ has the Lax pair representation.
Up to the symmetry
$(\xi,v)\mapsto(-\xi,-v)$,
the billiard mapping is equivalent to the equation
$$
\tilde L(\lambda)=M(\lambda)L(\lambda)M^{-1}(\lambda),
$$
where:
\begin{gather*}
L(\lambda)=\left(
 \begin{array}{cc}
q_{\lambda}(\xi,v) & q_{\lambda}(v,v)-\sigma\\
-q_{\lambda}(\xi,\xi)+1 & -q_{\lambda}(\xi,v)
\end{array}\right),
    \\
M(\lambda)=\left(\begin{array}{cc}
 \sigma\lambda-(v,Av)\lambda+2(\xi,Av)\mu &
 2\sigma(\xi,Av)\lambda-\sigma\mu+(v,Av)\mu
\\
 -2(\xi,Av)\lambda & \sigma\lambda-(v,Av)\lambda
\end{array}\right),
\end{gather*}
with $\tilde L(\lambda)$ depending on $\tilde \xi, \tilde v$ in the same way as $L(\lambda)$ on $\xi,v$,
and
$$
q_{\lambda}(\xi,\eta)=\frac{\xi_1\eta_1}{a-\lambda}+\frac{\xi_2\eta_2}{b-\lambda}.
$$

The billiard trajectories have the following geometrical properties.

Each segment of a given billiard trajectory within $\mathcal{E}$ is an arc of an ellipse with the centre at the coordinate origin \cite{ArnoldMMM}.
Moreover, as proved in \cite{Fed2001}, all these ellipses are touching the same pair of conics, confocal with $\mathcal{E}$.

If we denote the family of conics confocal with $\mathcal{E}$ as follows:
\begin{equation}\label{eq:confocal}
\mathcal{C}_{\lambda}\ :\ \frac{x^2}{a-\lambda}+\frac{y^2}{b-\lambda}=1,
\end{equation}
then the parameters of the caustics of a given trajectory are the roots of the characteristic polynomial:
\begin{equation}\label{eq:poly}
p(\lambda)=(\lambda-a)(\lambda-b)\det L(\lambda).
\end{equation}

\section{Topological properties}\label{sec:topology}
The aim of this section is to give the topological characterisation of the isoenergy manifolds.

In the next proposition, we summarise some geometrical properties of the system.

\begin{prop}\label{prop:kaustike}
Consider a given billiard trajectory within the ellipse $\mathcal{E}$ (\ref{eq:E}) with the Hooke's potential (\ref{eq:potential}).
Let $E$ be the total energy corresponding to that trajectory and $\mathcal{C}_{\lambda_1}$, $\mathcal{C}_{\lambda_2}$ its caustics, $\lambda_1\le\lambda_2$.
Then:
\begin{itemize}
\item 
$\lambda_1+\lambda_2=a+b-\frac{2}{\sigma} E$;
\item
$\mathcal{C}_{\lambda_1}$ is an ellipse and it contains $\mathcal{E}$, i.e. $\lambda_1\le0$;
\item
if a segment of the trajectory contains a focus of $\mathcal{E}$, then the next segment contains the other focus.
$\mathcal{C}_{\lambda_2}$ is then the degenerate conic: $\lambda_2=b$;
\item
if a segment of the trajectory itersects the segment containing the foci, than each segment intersects it as well.
$\mathcal{C}_{\lambda_2}$ is then a hyperbola or the degenerate conic that coincides with the $y$-axis: 
$b<\lambda_2\le a$;
\item
if a segment of the trajectory does not itersect the segment containing the foci, than none of the segments intersect it.
$\mathcal{C}_{\lambda_2}$ is then an ellipse within $\mathcal{E}$: $0\le\lambda_2<b$.
\end{itemize}
\end{prop}
\begin{proof}
The leading coefficient of the characteristic polynomial $p(\lambda)$ (\ref{eq:poly}) is equal to $\sigma$, and the coefficient multiplying $\lambda$ is $v^2+\sigma\xi^2-\sigma(a+b)=2E-\sigma(a+b)$, from where the first statment follows.

One of the caustics will be the ellipse containing the extensions of the billiard segments.
Since the billiard impacts are off $\mathcal{E}$, that caustic need to satisfy the second statement.

The other caustic is either an ellipse inscribed in the trajectory, a hyperbola, or a degenerate conic, which proves the rest of the proposition.
\end{proof}

We will refer to the conic $\mathcal{C}_{\lambda_1}$ from Proposition \ref{prop:kaustike} as \emph{the outer caustic} of the given billiard trajectory and to $\mathcal{C}_{\lambda_2}$ as \emph{the inner caustic}.
Proposition \ref{prop:kaustike} shows that the inner caustic has the properties analoguous to the propeties of the unique caustic in the case of the elliptical billiard without potential.
The billiard with the Hooke's potential has the famous focal property as well.

The phase space of the elliptical billiard is:
$$
\mathcal{M}=\{
(\xi,v) \mid \xi \text{ is within } \mathcal{E} \text{ and } v\in T_{\xi}\mathbf{R}^2
\}/\sim
$$
with
\begin{gather*}
(\xi,v_1)\sim(\xi,v_2)
\\
\Longleftrightarrow
\\
\xi\in\mathcal{E}
\text{ and }
v_1+v_2\in T_{\xi}\mathcal{E}
\text{ and }
v_1-v_2\perp T_{\xi}\mathcal{E}.
\end{gather*}

Note that $\mathcal{M}$ contains trajectories that have empty intersection with the boundary -- these are trajectories with the outer caustic within $\mathcal{E}$.
A trajectory with the caustics $\mathcal{C}_{\lambda_1}$, $\mathcal{C}_{\lambda_2}$ will have a non-empty intersection with $\mathcal{E}$ if and only if the outer caustic contains $\mathcal{E}$ and the inner caustic is either ellipse within $\mathcal{E}$ or a hyperbola, which is equivalent to
 $
 \lambda_1\lambda_2\le0,
 $
i.~e.~the constant term of the characteristic polynomial (\ref{eq:poly}) need to be non-positive:
$$
(v_2\xi_1-v_1\xi_2)^2-(bv_1^2+av_2^2)-\sigma(b\xi_1^2+a\xi_2^2)+ab\sigma\le0.
$$

So, we will consider the following subset of $\mathcal{M}$:
$$
\mathcal{M}_B=\{ (\xi,v)\in\mathcal{M} \mid  (v_2\xi_1-v_1\xi_2)^2-(bv_1^2+av_2^2)-\sigma(b\xi_1^2+a\xi_2^2)+ab\sigma\le0 \}.
$$

Denote by $\mu$ be the following map:
$$
\mu\ :\ \mathcal{M}_B\to\mathbf{R}^2,
\quad
(\xi,v)\mapsto(E,\lambda_2),
$$
with
\begin{gather*}
E(\xi,v)=\frac{v^2}2+\frac{\sigma\xi^2}2,
\\
\lambda_2(\xi,v)=
\frac{\sigma(a+b)-v^2-\sigma\xi^2+\sqrt{(\sigma(a+b)-v^2-\sigma\xi^2)^2-4\sigma p(0)}}
{2\sigma},
\\
p(0)=(v_2\xi_1-v_1\xi_2)^2-(bv_1^2+av_2^2)-\sigma(b\xi_1^2+a\xi_2^2)+ab\sigma.
\end{gather*}
  
It maps each point of $\mathcal{M}_B$ to the pair $(E,\lambda_2)$ corresponding to the total energy and the parameter of the inner caustic.
From Proposition \ref{prop:kaustike}, it follows that the image of $\mu$ is the set:
$$
\left\{(E,\lambda_2)\mid 0\le\lambda_2\le a\text{ and }\lambda_2+\frac{2}{\sigma} E\ge a+b\right\},
$$
which is shown in Figure \ref{fig:bif}.

\begin{figure}[h]
\centering
\begin{pspicture}(-1.5,-1.5)(5.5,3.5)

\psset{linecolor=black,fillstyle=solid}


\psline{->}(-1,-0.5)(-1,3)
\rput(-1.3,3){$\lambda_2$}

\psline(-1.09,0)(-0.91,0)
\rput(-1.3,0){$0$}

\psline(-1.09,1)(-0.91,1)
\rput(-1.3,1){$b$}
 
\psline(-1.09,2.5)(-0.91,2.5)
\rput(-1.3,2.5){$a$}


\psline{->}(-0.5,-1)(5,-1)
\rput(5,-1.3){$E$}

\psline(0,-1.09)(0,-0.91)
\rput(0,-1.4){$\frac{\sigma b}{2}$}

\psline(1.5,-1.09)(1.5,-0.91)
\rput(1.5,-1.4){$\frac{\sigma a}{2}$}

\psline(2.5,-1.09)(2.5,-0.91)
\rput(2.5,-1.4){$\frac{\sigma(a+b)}{2}$}


\pspolygon[linecolor=white,linestyle=none,fillstyle=solid,fillcolor=gray!50](5,2.5)(0,2.5)(2.5,0)(5,0)

\psline[linewidth=1.2pt](5,2.5)(0,2.5)
\psline[linewidth=1.2pt,linecolor=gray](0,2.5)(2.5,0)
\psline[linewidth=1.2pt](2.5,0)(5,0)

\psline[linestyle=dashed,linewidth=1.2pt](1.5,1)(5,1)

\psline[linestyle=dotted,dotsep=1pt,linecolor=gray,fillstyle=none](-1,2.5)(0,2.5)(0,-1)
\psline[linestyle=dotted,dotsep=1pt,linecolor=gray,fillstyle=none](-1,1)(1.5,1)(1.5,-1)
\psline[linestyle=dotted,dotsep=1pt,linecolor=gray,fillstyle=none](-1,0)(2.5,0)(2.5,-1)

\pscircle*(0,2.5){0.05}
\pscircle*(1.5,1){0.05}
\pscircle*(2.5,0){0.05}

\end{pspicture}
\caption{Bifurcation diagram for the elliptical billiard with the Hooke's potential.}\label{fig:bif}
\end{figure}
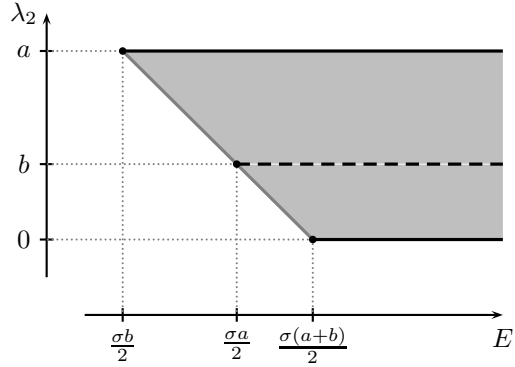 

\begin{thm}\label{th:fomenko}
The isoenergy manifolds for the billiard within the ellipse $\mathcal{E}$ with the Hooke's potential (\ref{eq:potential})
are shown in Figure \ref{fig:fom}.

There $\mathbf{A}$ and $\mathbf{B}$ are standard Fomenko atoms: $\mathbf{A}$ is a level set containing a single closed orbit, while $\mathbf{B}$ denotes a singular level set containing one closed orbit and two separatrices with homoclinic trajectories. 

By $\mathbf{T}$ we denoted a torus containing only closed orbits, and by $\mathbf{8}$ a product of figure $8$ with the circle -- a non-standard singular level set that contains only closed orbits.
\end{thm}

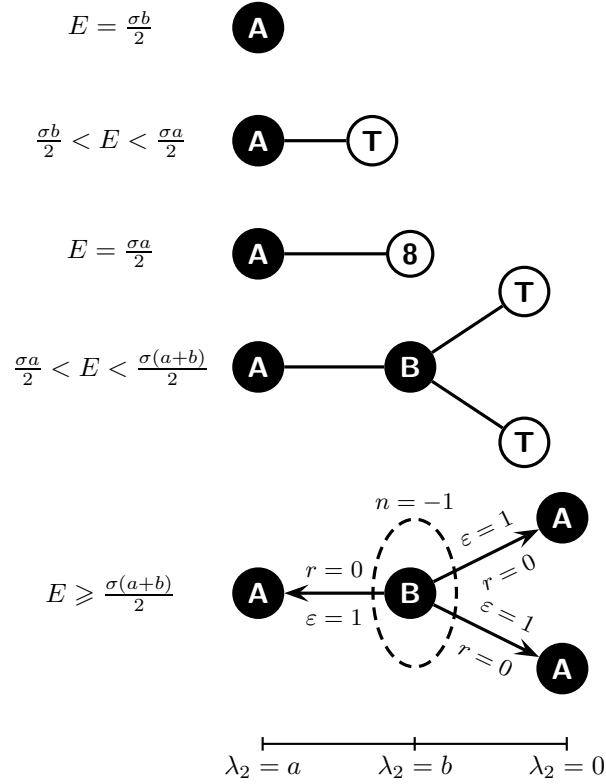
\begin{figure}[h]
\centering
\begin{pspicture}(-3,0)(4.5,10.5)


\psline(0,0.5)(4,0.5)

\psline(4,.41)(4,0.59)
\rput(4,0.2){$\lambda_2=0$}

\psline(2,.41)(2,0.59)
\rput(2,0.2){$\lambda_2=b$}

\psline(0,.41)(0,0.59)
\rput(0,0.2){$\lambda_2=a$}

\psset{
    fillstyle=solid,
    arrowsize=3pt 3,
    arrowlength=1.4,
    arrowinset=0.4,
    linewidth=1.2pt
    }

\rput(0,10){%
    \circlenode[fillcolor=black,linecolor=black]{1}{\color{white}\large\textbf{\textsf{A}}}
    }
\rput(-2,10){$E=\frac{\sigma b}{2}$}

\rput(0,8.5){%
    \circlenode[fillcolor=black,linecolor=black]{2}{\color{white}\large\textbf{\textsf{A}}}
    }
\rput(1.5,8.5){%
    \circlenode[fillcolor=white,linecolor=black]{3}{\large\textbf{\textsf{T}}}
    }
\ncline{-}{2}{3}    
\rput(-2,8.5){$\frac{\sigma b}{2}<E<\frac{\sigma a}{2}$}

\rput(0,7){%
    \circlenode[fillcolor=black,linecolor=black]{4}{\color{white}\large\textbf{\textsf{A}}}
    }
\rput(2,7){%
    \circlenode[fillcolor=white,linecolor=black]{5}{\large\textbf{\textsf{8}}}
    }
\ncline{-}{4}{5}    
\rput(-2,7){$E=\frac{\sigma a}{2}$}

\rput(0,5.5){%
    \circlenode[fillcolor=black,linecolor=black]{6}{\color{white}\large\textbf{\textsf{A}}}
    }
\rput(2,5.5){%
    \circlenode[fillcolor=black,linecolor=black]{7}{\color{white}\large\textbf{\textsf{B}}}
    }
\rput(3.5,4.5){%
    \circlenode[fillcolor=white,linecolor=black]{8}{\large\textbf{\textsf{T}}}
    }
\rput(3.5,6.5){%
    \circlenode[fillcolor=white,linecolor=black]{9}{\large\textbf{\textsf{T}}}
    }       
\ncline{-}{6}{7}
\ncline{-}{8}{7}    
\ncline{-}{9}{7}
\rput(-2,5.5){$\frac{\sigma a}{2}<E<\frac{\sigma(a+b)}{2}$}

\rput(0,2.5){%
    \circlenode[fillcolor=black,linecolor=black]{10}{\color{white}\large\textbf{\textsf{A}}}
    }
\rput(2,2.5){%
    \circlenode[fillcolor=black,linecolor=black]{11}{\color{white}\large\textbf{\textsf{B}}}
    }
\rput(4,1.5){%
    \circlenode[fillcolor=black,linecolor=black]{12}{\color{white}\large\textbf{\textsf{A}}}
    }
\rput(4,3.5){%
    \circlenode[fillcolor=black,linecolor=black]{13}{\color{white}\large\textbf{\textsf{A}}}
    }       
\ncline{<-}{10}{11}
\naput{\small $r=0$}
    \nbput{\small $\varepsilon=1$}
\ncline{->}{11}{12}
\nbput[nrot=-25.64,npos=0.6]{\small $r=0$}
    \naput[nrot=-25.64,npos=0.6]{\small $\varepsilon=1$}
\ncline{->}{11}{13}
\nbput[nrot=25.64,npos=0.6]{\small $r=0$}
    \naput[nrot=25.64,npos=0.6]{\small $\varepsilon=1$}
\rput(-2,2.5){$E\ge\frac{\sigma(a+b)}{2}$}   

\psellipse[fillstyle=none,linestyle=dashed](2,2.5)(0.57,1)  
\rput(2,3.7){\small $n=-1$}

\end{pspicture}

\caption{The Fomenko graphs for the elliptical billiard with the Hooke's potential.}\label{fig:fom}
\end{figure} 

\begin{proof}
The lowest energy level that contains billiard trajectories is $E=\frac{\sigma b}{2}$.
There, only one billiard trajectory exists -- the one placed along $y$-axis.
At the touching points with $\mathcal{E}$, the velocity equals $0$.

For $E\in(\frac{\sigma b}{2},\frac{\sigma a}{2})$, the inner caustic is always a hyperbola.
There is one Liouville torus corresponding to each non-singular level set.
The isoenergy manifold contains two singular level sets.
The first one corresponds to the degenerate inner caustic $\mathcal{C}_a$ and it contains only one orbit:
the trajectory is placed along $y$-axis, but now the particle hits the boundary with non-zero velocity.
That level set is represented by the Fomenko atom $\mathbf{A}$.

The second singular level set corresponds to the inner caustic with the parameter $\lambda_2=a+b-\frac{2}{\sigma}E$ and the outer caustic $\mathcal{E}$.
The orbits on that level set correspond to the ellipses with the centre at the origin, which are inscribed in $\mathcal{E}$ and touching $\mathcal{C}_{\lambda_2}$.
Among them, there are two ellipses that degenerate to the segments containing the origin with the endpoints at the intersection points of $\mathcal{E}$ and $\mathcal{C}_{\lambda_2}$.
The level set is a torus and it is represented by the non-standard atom $\mathbf{T}$.

For $E=\frac{\sigma a}{2}$, the inner caustic corresponding to each non-singular level set is again a hyperbola.
The non-standard singular level set, represented by the vertex $\mathbf{8}$, corresponds to the degenerate caustic $\mathcal{C}_{b}$.
The orbits on that level set correspond to the ellipses with the centre at the origin, which are inscribed in $\mathcal{E}$ and contain the foci.
One of these ellipses degenerates to the big diameter of $\mathcal{E}$.
The level set is isomorphic to the product of the figure $8$ with the circle.
The orbit on the self-intersection part of the set corresponds to the motion along the $x$-axis. 

For $E\in(\frac{\sigma a}{2},\frac{\sigma(a+b)}{2})$, there is one Liouville torus in each of the level sets that correspond to hyperbolas as inner caustics, and two Liouville tori if the inner caustic is an ellipse.
The non-standard singular level set contains closed orbits -- the trajectories are ellipses inscribed in $\mathcal{E}$ and circumscribed about the inner caustic.
That set is composed of two tori -- each one corresponds to one winding direction about the origin.

When $E\ge\frac{\sigma(a+b)}{2}$ there are no non-standard singular level sets.
The two $\mathbf{A}$-atoms on the right hand side of the Fomenko atom correspond to the limit case when the inner caustic coincides with $\mathcal{E}$.
When the energy approches infinity, the billiard segments become straight, so the numerical topological invariants will be the same as in the case of the billiard without potential.
\end{proof}

\begin{rem}
The bifurcation set of the system is the union of three half-lines:
$$
\lambda_2=a,\ E\ge\frac{\sigma b}{2},
\qquad
\lambda_2=b,\ E\ge\frac{\sigma a}{2},
\qquad
\lambda_2=0,\ E\ge\frac{\sigma (a+b)}{2},
$$
see Figure \ref{fig:bif}.
\end{rem}

\begin{rem}
The segments $\lambda_2=a+b-\frac2{\sigma}E\in(0,b)\cup(b,a)$ does not belong to the ciritical set.
They correspond to the value $\lambda_1=0$, that is to the case when the outer caustic $\mathcal{C}_{\lambda_1}$ coincides with the billiard border $\mathcal{E}$.
All trajectories mapped to those segments are ellipses.

For $\lambda_2\in(b,a)$, the trajectories are ellipses inscribed in $\mathcal{E}$ and circumscribed about the hyperbola $\mathcal{C}_{\lambda_2}$.
Among them, there are two degenerate ellipses -- segments containing the origin with the endpoints at the intersections of $\mathcal{E}$ and $\mathcal{C}_{\lambda_2}$.
Each non-degenerate ellipse is covered two orbits, corresponding to the clockwise and the counterclockwise motion about the origin.

For $\lambda_2\in(0,b)$, the trajectories are ellipses inscribed in $\mathcal{E}$ and circumscribed about the ellipse $\mathcal{C}_{\lambda_2}$.
There are two connected components of the level set, one containing the clockwise orbits, another counterclockwise ones.

The corresponding level sets are non-standard atoms $\mathbf{T}$.
Saturated neighbourhoods of these atoms in the isoenergy manifolds are shown in Figure \ref{fig:atomT}.
\end{rem}

\begin{rem}
Point $\lambda_2=b$, $E=\frac{\sigma a}{2}$ belongs to the critical set.
The trajectories are ellipses inscribed in $\mathcal{E}$ and containing the foci.
Only one trajectory is a critical orbit -- the one placed along $x$-axis.

The corresponding level set is a non-standard atom $\mathbf{8}$.
Its saturated neighbourhood is shown in Figure \ref{fig:atom8}.
\end{rem}

\begin{figure}[h]
\begin{minipage}{0.45\textwidth}
\centering
\begin{pspicture}(-2,-2)(3,2)

\psset{linecolor=black,fillstyle=solid}

\pscircle[fillcolor=gray!50,linecolor=white](0,0){2}

\pscircle[fillcolor=white,linecolor=black](0,0){1}

\psset{linewidth=0.8pt,linecolor=gray,fillstyle=none}

\rput(2.5,0){$\times\mathbf{S}^1$}

\end{pspicture}
\caption{Neighbourhood of $\mathbf{T}$.}\label{fig:atomT}
\end{minipage}
\begin{minipage}{0.45\textwidth}
\centering
\begin{pspicture}(-2,-2)(3,2)

\psset{linecolor=black,fillstyle=solid}
\pscircle[fillcolor=gray!50,linecolor=white](0,0){2}

\psline[linecolor=white](-1,0)(1,0)

\psplot{-1}{1}{x x mul 1 x x mul sub mul sqrt}
\psplot{-1}{1}{x x mul 1 x x mul sub mul sqrt neg}

\psset{fillstyle=none}
\psplot{-1}{1}{x x mul 1 x x mul sub mul sqrt}
\psplot{-1}{1}{x x mul 1 x x mul sub mul sqrt neg}

\pscircle*(0,0){0.05}

\psset{linewidth=0.8pt,linecolor=gray}

\psplot{-1.16877}{1.16877}{x x mul 1 x x mul sub mul 0.5 add sqrt}

\psplot{-1.16877}{1.16877}{x x mul 1 x x mul sub mul 0.5 add sqrt neg}

\rput(2.5,0){$\times\mathbf{S}^1$}

\end{pspicture}
\caption{Neighbourhood of $\mathbf{8}$.}\label{fig:atom8}
\end{minipage}
\end{figure}

Along with the bifurcation diagram, it is useful to construct the \emph{bifurcation complex}, that is the topological space, whose points are connected components of the level sets in the Liouville foliation, with the natural factor-topology \cite{BBM2010}.
The bifurcation complex for the elliptical billiard with the Hooke's potential is a cell complex composed of three cells with a common half-line, as shown in Figure \ref{fig:complex}.

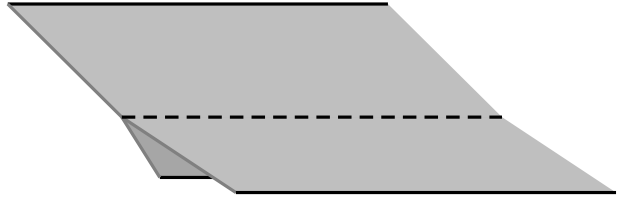
\begin{figure}[h]
\begin{pspicture}(-.5,-.5)(8,3.5)

\psset{linecolor=black,fillstyle=none}


\pspolygon[linecolor=white,linestyle=none,fillstyle=solid,fillcolor=gray!50](5,2.5)(0,2.5)(1.5,1)(6.5,1)

\pspolygon[linecolor=white,linestyle=none,fillstyle=solid,fillcolor=gray!70](6.5,1)(1.5,1)(2,0.2)(5,0.2)

\psline[linewidth=1.2pt](2,0.2)(5,0.2)

\pspolygon[linecolor=white,linestyle=none,fillstyle=solid,fillcolor=gray!50](6.5,1)(1.5,1)(3,0)(8,0)

\psline[linewidth=1.2pt](5,2.5)(0,2.5)
\psline[linewidth=1.2pt,linecolor=gray](0,2.5)(1.5,1)(3,0)
\psline[linewidth=1.2pt,linecolor=gray](1.5,1)(2,0.2)
\psline[linewidth=1.2pt,linestyle=dashed](1.5,1)(6.5,1)
\psline[linewidth=1.2pt](3,0)(8,0)

\end{pspicture}
\caption{Bifurcaton complex.}\label{fig:complex}
\end{figure}

\begin{cor}\label{cor:stability}
Critical periodic billiard trajectories along $y$-axis and the limit periodic trajectories along the ellipse $\mathcal{E}$ are stable, while
critical periodic trajectories along $x$-axis are unstable.
\end{cor}

\begin{proof}
The statement follows from Theorem \ref{th:fomenko}, since the trajectories along $y$-axis they correspond to $\mathbf{A}$-atoms for $\lambda_2=a$, the limit trajectories along $\mathcal{E}$ to $\mathbf{A}$-atoms for $\lambda_2=0$, and the critical trajectories along $x$-axis are contained in the $\mathbf{B}$ and $\mathbf{8}$ atoms from $\lambda_2=b$.

Using \cite{BBM2010}, we can get that conclusion from the bifurcation complex in Figure \ref{fig:complex}, since the points corresponding to $\lambda_2=a$ and $\lambda_2=0$ are on the border of the complex, while inner points of the complex correspond to $\lambda_2=b$.
\end{proof}

\begin{rem}
The stability of periodic orbits of general billiards with potential is studied in \cite{Dullin1998}.
\end{rem}

\subsection*{Acknowledgments} 
{The author is grateful to the referee for great comments and questions, which led to the significant improvement of the manuscript.

The research which has led to this paper was partially
supported by the Serbian Ministry of Education and Science (Project
no.~174020: \emph{Geometry and Topology of Manifolds and Integrable
Dynamical Systems})
and by grant no.~FL120100094 from the Australian Research Council.}

\end{document}